\documentclass[12pt,reqno]{amsart}
\usepackage[margin=1in]{geometry}
\usepackage{amsmath, amsthm, amssymb}
\usepackage[colorlinks=true, pdfstartview=FitV, linkcolor=blue,citecolor=blue, urlcolor=blue]{hyperref}
\usepackage[abbrev,lite,nobysame]{amsrefs}
\usepackage{times}
\usepackage[usenames,dvipsnames]{color}
\usepackage{graphics,graphicx,pspicture,graphpap}
\usepackage{mathtools}

\mathtoolsset{showonlyrefs=true}

\def\eps{\varepsilon}
\def\e{{\rm e}}
\def\dd{{\rm d}}
\def\ddt{{\frac{\dd}{\dd t}}}
\def\uu {\boldsymbol{u}}
\def\Id{{\rm Id}}
\def\spt{{\rm spt}}
\def\RR {\mathbb{R}}

\def\NN {\mathbb{N}}
\def\PP {\mathbb{P}}
\def\MA{\mathcal{P}^\nu}
\def\EE {\mathbb{E}}
\def\TT {{\mathbb T}}
\def\MM {{\mathfrak P}}
\def\BB {{\mathcal B}}
\def\XX {{\mathcal X}}
\def\YY {{\mathcal Y}}
 \def\mmu {\boldsymbol{\mu}}
\def\grad{{\nabla}}
\def\de{{\partial}}
\newcommand{\norm}[1]{\left\lVert #1 \right\rVert}
\newcommand{\abs}[1]{\left\vert#1\right\vert}
\newcommand{\set}[1]{\left\{#1\right\}}
\newcommand{\Naturals}{\mathbb N}
 
\newcommand{\inp}[1]{\left\langle #1 \right\rangle}
\newcommand{\Real}{\mathbb R}

\newcommand{\Torus}{\mathbb T}
\newcommand{\Integers}{\mathbb Z}
\newcommand{\Integer}{\mathbb Z}

\def \l {\langle}
\def \r {\rangle}

\newtheorem{proposition}{Proposition}[section]
\newtheorem{theorem}[proposition]{Theorem}
\newtheorem{corollary}[proposition]{Corollary}
\newtheorem{lemma}[proposition]{Lemma}
\theoremstyle{definition}
\newtheorem{definition}[proposition]{Definition}
\newtheorem{remark}[proposition]{Remark}

\numberwithin{equation}{section}

\title[Invariant measures for passive scalars in the small noise inviscid limit]{Invariant measures for passive scalars in the small noise inviscid limit}

\author[J. Bedrossian, M. Coti Zelati, and N. Glatt-Holtz]{Jacob Bedrossian, Michele Coti Zelati, and Nathan Glatt-Holtz}

\address{Department of Mathematics, University of Maryland, College Park, MD 20742, USA}
\email{jacob@cscamm.umd.edu}
\email{micotize@umd.edu}

\address{Department of Mathematics, Virginia Tech, Blacksburg, VA 24061, USA}
\email{negh@vt.edu}



\begin{document}

\begin{abstract}
We consider a class of invariant measures for a passive scalar $f$ driven by an incompressible velocity field
$\uu$, on a $d$-dimensional periodic domain, satisfying 
$$
\de_t f + \uu \cdot \grad f = 0, \qquad f(0)=f_0.
$$ 
The measures are obtained as limits of stochastic viscous perturbations. We prove that the span of the
$H^1$ eigenfunctions of the operator $\uu\cdot\grad$ contains the support of these measures.
We also analyze several explicit examples: when $\uu$ is a shear flow or a relaxation enhancing flow 
(a generalization of weakly mixing), we can characterize the limiting measure uniquely and compute its
covariance structure. We also consider the case of two-dimensional cellular flows, for which
further regularity properties of the functions in the support of the measure can be deduced.

The main results are proved with the use of spectral theory results, in particular the RAGE theorem, 
which are used to characterize large classes of orbits of the inviscid problem that are growing in $H^1$.  
\end{abstract}

\maketitle


\section{Introduction}
We consider the small noise inviscid limit of a class of stochastically forced linear drift-diffusion equations
of the form
\begin{equation}\label{eq:eq1}
\dd f + \left(\uu \cdot \grad f - \nu \Delta f\right) \dd t = \sqrt{\nu} \, \Psi\, \dd W_t, \qquad f(0)=f_0,
\end{equation} 
evolving on a $d$-dimensional periodic domain. 
Here $\uu$ is a fixed Lipschitz-continuous, divergence-free vector field, $\nu>0$ 
is the diffusivity parameter and  $\Psi \dd W_t$ represents a white in time, spatially colored Gaussian noise 
(see \eqref{eq:SPDEviscous} below for the full definition). 
We will always consider mean-zero initial data and forcing,
which immediately implies that 
$$
\int_{\TT^d} f(t,x) \dd x = 0
$$ 
for all $t\geq 0$.

It is a classical result \cite{DZ96} that there exists a unique Gaussian invariant measure $\mu_\nu$ 
associated to the Markov semigroup generated by \eqref{eq:eq1}. Due to the balance between 
diffusion and noise, it is possible to show that the sequence $\{\mu_\nu\}_{\nu\in (0,1]}$ 
converges, up to subsequences, to an invariant measure $\mu_0$ of the inviscid deterministic equation (in general, there might be more than one limit point) 
\begin{equation}\label{eq:eq2}
\de_t f + \uu \cdot \grad f = 0, \qquad f(0)=f_0.
\end{equation}
In this article, we characterize the support of these measures $\mu_0$ in terms of the spectral properties
of the operator $\uu\cdot \grad$. Specifically, we prove the following theorem.

\begin{theorem}\label{thm:ee}
Let $\mu_0$ be an invariant measure for \eqref{eq:eq2}, obtained in the small noise inviscid 
limit from \eqref{eq:eq1}. Define
\begin{equation}\label{eq:EE}
E= \overline{{\rm{span}}\big\{\varphi\in H^1: \uu\cdot \nabla\varphi=i \lambda \varphi, \ \lambda\in \RR \big\}}^{L^2}.
\end{equation}
Then $\mu_0(H^1\cap E)=1$. In particular, $\spt(\mu_0)\subset E$.
\end{theorem}
Extensions to more general linear problems are also covered by our approach (e.g. compact manifolds without boundaries and more general dissipation, such as 
fractional or inhomogeneous diffusion). 

In addition, we consider 
several concrete examples where we are able to characterize the subspace $E$  in \eqref{eq:EE} and/or to take advantage 
of properties of $\uu\cdot \grad - \nu \Delta$ to obtain a more detailed picture of the support of $\mu_0$.
The theorem above is derived as a consequence 
of a rigidity result involving a uniform time-average growth of Sobolev norms of solutions to \eqref{eq:eq2}
(see Theorem \ref{thm:rigidity}). This criterion, verified for linear problems, could  in principle be
used to deduce similar conclusions for nonlinear systems (see Remark \ref{rmk:nonlin} below).

The idea of balancing diffusion and noise by scaling with the parameter $\nu$ was introduced
 in the context of the two-dimensional Navier-Stokes 
and Euler equations in a periodic domain in \cites{K04,KS12}, and later extended to other systems \cites{K10,KS04}. 
The resulting invariant measures  
are expected to give some information about the generic, long-time dynamics of the inviscid 
systems in weak topologies, that is, taking into account the possibility of
infinite-dimensional effects such as mixing and inviscid damping \cites{BM13, GSV13}. 
However, for most nonlinear problems, for example the 2D Euler equations, we currently 
do not have much explicit information about these measures \cites{KS12, GSV13, MP14}. 

The mixing of passive scalars at high Peclet number ($\nu \rightarrow 0$ in \eqref{eq:eq1}) is a subject of enduring interest in applied mathematics (see e.g. \cites{RY83,BBG01,LTD11,AT11,S13} and the references therein), which provides a clear motivation for studying \eqref{eq:eq1}.  
Another motivation for studying passive scalars  lies in the
fact that, unlike nonlinear problems, spectral properties of $\uu\cdot\grad$ provide precise information about the 
long-time behavior of  \eqref{eq:eq2} -- an idea that can be traced back to the seminal work \cite{KvN32},
where weakly mixing flows are identified with dynamical systems with purely continuous spectra.
This additional information will allow us to prove Theorem \ref{thm:ee} and confirm the intuition that the inviscid invariant measures constructed 
in the manner described above should mostly retain information about the long-time dynamics of 
the large scales in the solutions, rather than information about the ``enstrophy'' in the small scales. 

The intuitive idea of our result is as follows: if $\uu \cdot \grad$ mixes the scalar, then small scales are 
created and then  rapidly annihilated by the dissipation on time scales faster than the natural $O(\nu^{-1})$ 
scale. This mixing-enhanced dissipation effect and related mechanisms have been studied in several 
works, e.g. \cites{BW13, BHN05, BMV14, CKRZ08, Zlatos2010, HKP14} and the references above (see also Remarks \ref{rmk:ga}-\ref{rmk:fried} for 
discussions about other types of small noise limits). 
In \cite{CKRZ08}, the authors characterized a 
special class of flows, referred to therein as \emph{relaxation enhancing}, which are precisely the flows 
with no $H^{1}$ eigenfunctions (a strictly larger class than weakly mixing flows). The authors showed 
that in this case, the deterministic problem
\begin{align}  
\de_t f + \uu \cdot \grad f -\nu\Delta f= 0, \qquad f(0)=f_0.
\end{align}
dissipates $L^2$ density faster than $O(\nu^{-1})$. 

In our context, Theorem \ref{thm:ee} shows that in the case of relaxation enhancing flows 
the only invariant measure produced by small noise limits is a point mass at zero. 
Note that this is in sharp contrast to the case of 2D Navier-Stokes to Euler limit, where it is known that the resulting inviscid 
measures cannot collapse to a single point; namely, the possibility of $\mu_0$ being a Dirac mass on
a steady state is ruled out by conservation of energy and enstrophy \cite{KS12}.   
One of the fundamental differences between the situation considered here and the 2D Navier-Stokes
equations is the lack of an $H^{-1}$ balance for the solution to \eqref{eq:eq1}. Indeed, $f$ is the analogue 
of the vorticity, and the energy balance for 2D Navier-Stokes is precisely an $H^{-1}$ balance for the vorticity.  
In the linear problem \eqref{eq:eq2}, the measures can certainly reduce to a point mass: the most 
apparent example is given by weakly mixing flows, and more generally relaxation enhancing flows, but this 
can also happen for simple shear flows depending on the structure of the noise (see Section \ref{sec:examples}).   

\subsection*{Plan of the paper}
In Section \ref{sec:inviscid} we consider the inviscid problem \eqref{eq:eq2}, and  prove a general
result on the growth of the $H^1$ norm of solutions with initial data in $E^\perp$, the orthogonal
complement of $E$ as defined in \eqref{eq:EE}. As a direct consequence, we deduce that invariant
measures for \eqref{eq:eq2} with finite $H^1$ moment are supported on $E$. 
We proceed with the construction of such measures in Section \ref{sec:section3}, via a small noise inviscid
limit of invariant measures for \eqref{eq:eq1}. We also discuss statistically stationary
solutions, and highlight the properties preserved in the limit as $\nu\to 0$. Finally, Section \ref{sec:examples}
is devoted to explicit examples of fluid flows for which the invariant measures can be better
characterized: in the case of relaxation enhancing flows and shear flows, 
the covariance operator of the (unique) Gaussian invariant measure can be computed explicitly, while
for cellular flows further regularity properties are observed to hold.

\subsection*{General notation}
 Throughout the paper, $c$ will denote a \emph{generic} positive constant,  
whose value may change from line to line in a given estimate. In the 
same spirit, $c_0,c_1,\ldots$ will denote fixed constants appearing in the course of proofs 
or estimates, which have to be referred to specifically. Given a Banach space $X$, $\BB(X)$ 
will stand for the Borel $\sigma$-algebra on $X$,
$\MM(X)$ for the set of Borelian probability measures on $X$ and $M_b(X)$ (resp. $C_b(X)$) for the space
of bounded measurable (resp. continuous) real-valued functions on $X$. We call $\RR^+=[0,\infty)$.

\subsection*{Function spaces}
Let $d\geq 2$ be a positive integer. Throughout  the article, $\TT^d=[0,2\pi]^d$ will denote the $d$-dimensional
torus and  all the real-valued functions on $\TT^d$ will be tacitly assumed to be  mean-free. Accordingly, we will not
make a distinction between homogeneous and inhomogeneous spaces. For $p\in[1,\infty)$ the Lebesgue 
norm on 
$$
L^p=\left\{\varphi:\TT^d\to\RR^d, \ \int_{\TT^d}|\varphi(x)|^p\dd x<\infty, \ \int_{\TT^d}\varphi(x)\dd x=0\right\}
$$ 
is denoted by $\|\cdot \|_{L^p}$ (with the obvious changes for $p=\infty$), $\l\cdot,\cdot\r$ stands for the 
scalar product in $L^2$, while for $s \in \RR$ 
the homogeneous Sobolev norms on $H^s=H^s(\TT^d)$ are denoted as usual 
by $\| \cdot\|_{H^s} = \| (-\Delta)^{s/2} \cdot \|_{L^2}$.
Without explicit reference, we will often make use
of the Poincar\'e inequality
\begin{equation}\label{eq:poinc}
\|\varphi \|_{L^2}\leq \frac{1}{\sqrt{\lambda_1}}\|\varphi\|_{H^1}, \qquad \varphi \in H^1,
\end{equation}
where $\lambda_1>0$ is the first eigenvalue of the Laplace operator, and whose eigenvalues
$\{\lambda_j\}_{j\in \NN}$ are well-known to form a monotonically increasing and divergent sequence.
The associated orthonormal Fourier basis will be denoted by $\{e_j\}_{j\in \NN}$, and $P_{\leq N}$
will indicate the projection onto the span of the first $N$ elements of this basis.

\section{The inviscid problem}\label{sec:inviscid}
For $x\in \TT^d$ and $t\geq 0$, we study in this section certain properties of solutions
to the inviscid transport equation
\begin{equation} \label{eq:inviscid}
\de_t f + \uu\cdot \grad f=0,\qquad f(0)=f_0.
\end{equation}
Here $\uu=\uu(x):\TT^d\to \RR^d$ is a \emph{given} Lipschitz, divergence-free, time-independent velocity vector field.
The goal here is to make precise the close relationship between the spectral properties of the operator
$\uu\cdot \grad$ and the invariant measures for the linear semigroup generated by \eqref{eq:inviscid}.

Any incompressible Lipschitz flow $\uu$ generates a volume measure-preserving transformation
$\Phi_t(x)$, defined through the differential equation
\begin{equation} \label{eq:lagr}
\ddt\Phi_t(x)=\uu(\Phi_t(x)), \qquad \Phi_0(x)=x.
\end{equation}
Existence and uniqueness of solutions to the above ODE \eqref{eq:lagr}, 
guaranteed by the assumptions on $\uu$, translate into analogous properties
for \eqref{eq:inviscid}. 
In particular, it is standard to infer that all solutions to \eqref{eq:inviscid} with $f_0 \in L^2$ satisfy 
\begin{align*}
f\in C_b(\RR; L^2)\cap W^{1,\infty}(\RR;H^{-1}) 
\end{align*}
and that \eqref{eq:inviscid} generates a one-parameter, strongly continuous, unitary group $\{S(t)\}_{t\in\RR}$ 
of linear solution operators $S(t):L^2\to L^2$ acting as
\begin{equation}
f_0\mapsto S(t)f_0=f(t),\qquad S(t)f_0(x) =  f_0(\Phi_{-t}(x)),
\end{equation}
fulfilling the group properties
\begin{align*} 
S(0)& =\Id_{L^2}, \qquad S(t+\tau)=S(t)S(\tau), \qquad S(t)^\ast= S(-t), \qquad \forall t,\tau\in \RR 
\end{align*} 
and satisfying the bound 
\begin{equation}\label{eq:globalest}
\sup_{t\in \RR} \norm{\de_t f(t)}_{H^{-1}} \leq \norm{f_0}_{L^2}\|\uu\|_{L^\infty}.
\end{equation}
Furthermore, if we assume the initial datum $f_0\in H^1$, then we have for some $c > 0$ independent of $\uu$, 
\begin{equation}
\|S(t)f_0\|_{H^1}\leq c\, \e^{\|\uu\|_{\mathrm{Lip}} \abs{t}}\|f_0\|_{H^1}, \qquad \forall t \in \RR.
\end{equation}

\subsection{Spectral properties of fluid flows} 
Inspired by the analysis of \cite{CKRZ08}, it is conceivable to expect that the set of $H^1$-eigenfunctions 
of the operator $\uu\cdot \grad$ plays an important role in the analysis of \eqref{eq:inviscid}. Specifically, this point of view will prove
very useful when \eqref{eq:inviscid} is recovered as an inviscid limit of viscous equations, 
providing information about important effects involving  anomalously fast 
dissipation. 
In this spirit, we define the closed subspace 
\begin{equation}\label{eq:smootheigen}
E= \overline{\mbox{span}\big\{\varphi\in H^1: \uu\cdot \nabla\varphi=i \lambda \varphi, \ \lambda\in \RR \big\}}^{L^2},
\end{equation}
generated by $H^1$-eigenfunctions of $\uu\cdot \grad$.
We can then write $L^2=E\oplus E^\perp$ and denote by 
\begin{equation}
\Pi_e:L^2\to E\quad \text{and}\quad \Pi_e^\perp:L^2\to E^\perp
\end{equation}
the respective orthogonal projections. It is not hard to see that the corresponding unbounded operator
\begin{equation}
L=i \uu\cdot \grad: D(L) \subset L^2\to L^2, \qquad D(L)=\left\{\varphi\in L^2: L\varphi\in L^2\right\},
\end{equation}
is closed, densely defined ($H^1 \subset D(L)$), self-adjoint, generates $\{S(t)\}_{t \in \Real}$ (i.e. $S(t) = \e^{iLt}$), and maps 
$D(L) \cap E$ to $E$, and therefore $D(L)\cap E^\perp$ to $E^\perp$
as well. Its restriction
\begin{align*} 
\widetilde{L}=L|_{E^\perp}: D(L) \cap E^\perp\to E^\perp,
\end{align*} 
is itself a closed, densely defined  and self-adjoint operator on $E^\perp$. 
Following the
approach of \cite{RS80-1}, we can therefore further split $E^\perp$ and
define the projection $\widetilde{\Pi}_p$ on the spectral subspace generated by the
pure point measure given by the spectral decomposition of $\widetilde{L}$.  Denote by $\widetilde{\Pi}_c$ 
the projection onto its orthogonal complement in $E^\perp$, that is, onto the orthogonal complement
of the eigenfunctions of $\widetilde{L}$.

More importantly, $\widetilde{L}$ has no $H^1$-eigenfunctions and is therefore, in the terminology of \cite{CKRZ08},
\emph{relaxation-enhancing} (see Definition \ref{def:Relax} below). We discuss relaxation-enhancing
flows in more detail below in Section \ref{sec:examples}. Of importance here are two results from \cite{CKRZ08}
on the behavior of time averages with respect to the linear unitary semigroup generated 
by $\widetilde{L}$, which we denote as
\begin{equation}
\widetilde{S}(t)=\e^{i\widetilde{L} t}:E^\perp\to E^\perp, \qquad \forall t\in \RR.
\end{equation}
The first one concerns the evolution of the continuous spectrum of $\widetilde{L}$, which we restate slightly for our setting. 
Its proof is based on the so-called RAGE theorem as in \cite{RS79-3}.

\begin{lemma}[\cite{CKRZ08}*{Lemma 3.2}]\label{lem:CKRZ1}
Let $\mathcal{K}\subset E^\perp$ be a compact set. For any $N,\sigma>0$, there exists $T_c=T_c(N,\sigma,\mathcal{K})$ such that 
for all $T\geq T_c$ and any $f_0\in \mathcal{K}$
\begin{equation}
\frac1T \int_0^T \|P_{\leq N}\widetilde{S}(t)\widetilde{\Pi}_cf_0\|^2_{L^2} \dd t \leq \sigma \|f_0\|^2_{L^2},
\end{equation}
where $P_{\leq N}$ is the projection onto the span of the first $N$ eigenfunctions of the Laplace operator.
\end{lemma}
Contrary to \cite{CKRZ08}, we have stated the result for a general compact set rather than a compact
subset of the unit sphere of $L^2$. By linearity, it is clear that the statements are equivalent. The 
crucial point here is that the choice of the time $T_c$ depends on the compact set $\mathcal{K}$ in a uniform
way rather than pointwise on $f_0\in \mathcal{K}$.
It is also worth mentioning that the above result is true for the operator $L$ and its spectral
projection $\Pi_c$ on $L^2$ as well, 
and does not require the absence of $H^1$-eigenfunctions. 
This is in contrast with the second lemma below, for which the fact
that $\widetilde{L}$ does not have $H^1$-eigenfunctions plays an essential role. It describes
the behavior of the point spectrum under $\widetilde{S}(t)$.

\begin{lemma}\cite{CKRZ08}*{Lemma 3.3}\label{lem:CKRZ2}
Let $\mathcal{K}\subset E^\perp$ be a compact set such that $0\notin \mathcal{K}$, and define 
$$
\mathcal{K}_1=\left\{\phi\in \mathcal{K}: \|\widetilde{\Pi}_p\phi\|_{L^2}\geq  \|\phi\|_{L^2}/2\right\}.
$$
For any $B>0$ there exists $N_p(B,\mathcal{K})$ and $T_p(B,\mathcal{K})$ such that for any $N\geq N_p$,
any $T\geq T_p$ and any $f_0\in \mathcal{K}_1$
\begin{equation}
\frac1T \int_0^T \|P_{\leq N}\widetilde{S}(t) \widetilde{\Pi}_pf_0\|^2_{H^1} \dd t \geq  B \|f_0\|^2_{L^2},
\end{equation}
where $P_{\leq N}$ is the projection onto the span of the first $N$ eigenfunctions of the Laplace operator.
\end{lemma}

\subsection{Invariant measures and their support: a general result}\label{sub:invmsr}
The main result of this section is a general statement about the average growth of the $H^1$-norm
of solutions to inviscid problems. 
The result is similar to ideas in \cite{CKRZ08,Zlatos2010}, where time-averaged norm growth of the inviscid problem is 
used to deduce enhanced dissipation of the viscous problem.  
It reads as follows.

\begin{theorem}\label{thm:rigidity}
Let $\mathcal{K}\subset H^1\cap E^\perp$  be a nonempty compact set in $L^2$ such that 
$0\notin \mathcal{K}$. Then the solution operator $S(t):L^2\to L^2$ for \eqref{eq:inviscid} satisfies
\begin{equation}\label{eq:growth}
\lim_{T\to\infty} \inf_{f_0\in \mathcal{K}}\frac1T\int_0^T \|S(t)f_0\|_{H^1}^2\dd t=\infty.
\end{equation}
\end{theorem}

\begin{proof}
We first notice that \eqref{eq:growth} only needs to be proven for $\widetilde{S}(t)$ since
we are assuming that $\mathcal{K}$ is a subset of $E^\perp$.
Fix $B> 0$ arbitrarily. We aim to find $T_0=T_0(B, \mathcal{K})>0$
such that for any $T\geq T_0$
\begin{equation}\label{eq:growth1}
\frac1T\int_0^T \|\widetilde{S}(t)f_0\|^2_{H^1} \dd t \geq B\|f_0\|_{L^2}^2, \qquad \forall f_0\in \mathcal{K}.
\end{equation}
The result will follow after noting that $\inf_{f_0 \in \mathcal{K}} \norm{f_0}_{L^2} > 0$ due to compactness and $0 \not\in \mathcal{K}$. 
Notice that it is important 
that $T_0$ depends on $\mathcal{K}$ but not on $f_0$.
Define the following subset of $\mathcal{K}$ (which is also compact in $L^2$),   
\begin{equation}
\mathcal{K}_1=\left\{\phi\in \mathcal{K}: \|\widetilde{\Pi}_p\phi\|_{L^2}\geq \|\phi\|_{L^2}/2\right\}.
\end{equation}
We consider two cases. 

\medskip

\noindent$\diamond$ \emph{Case 1:  $f_0 \in\mathcal{K}_1$}. 
By Lemma \ref{lem:CKRZ2}  there exist $N_p(B,\mathcal{K})$ and $T_p(B,\mathcal{K})$ such 
that for any $T\geq T_p$
\begin{equation}
\frac1T \int_0^T \|P_{\leq N_p}\widetilde{S}(t) \widetilde{\Pi}_pf_0\|^2_{H^1} \dd t \geq  4B \|f_0\|_{L^2}^2.
\end{equation}
Using Lemma \ref{lem:CKRZ1}, fix $T_c(N_p,B,\mathcal{K})$ such that
for all $T\geq T_c$
\begin{equation}
\frac1T \int_0^T \|P_{\leq N_p}\widetilde{S}(t)\widetilde{\Pi}_cf_0\|^2_{L^2} \dd t 
\leq \frac{B}{ \lambda_{N_p}}\|f_0\|_{L^2}^2,
\end{equation}
where $\lambda_{N_p}$ is the $N_p$-th eigenvalue of the Laplace operator on $\TT^d$.
Therefore, for any $T\geq \max\{T_p,T_c\}$ we have
\begin{align}
\frac1T\int_0^T \|\widetilde{S}(t)f_0\|^2_{H^1} \dd t 
&\geq \frac1T\int_0^T \|P_{\leq N_p}\widetilde{S}(t)f_0\|^2_{H^1} \dd t \\
&\geq \frac{1}{2T} \int_0^T \|P_{\leq N_p}\widetilde{S}(t)\widetilde{\Pi}_pf_0\|^2_{H^1}-\frac{1}{T} 
\int_0^T \|P_{\leq N_p}\widetilde{S}(t)\widetilde{\Pi}_cf_0\|^2_{H^1}\\
&\geq \frac{1}{2T} \int_0^T \|P_{\leq N_p}\widetilde{S}(t)\widetilde{\Pi}_pf_0\|^2_{H^1}-\frac{\lambda_{N_p}}{T} 
\int_0^T \|P_{\leq N_p}\widetilde{S}(t)\widetilde{\Pi}_cf_0\|^2_{L^2}\\
&\geq 2B\|f_0\|_{L^2}^2-B\|f_0\|_{L^2}^2=B\|f_0\|_{L^2}^2.
\end{align}
Therefore \eqref{eq:growth1} holds whenever $f_0\in \mathcal{K}_1$. 

\medskip

\noindent$\diamond$ \emph{Case 2:  $f_0 \notin\mathcal{K}_1$}.  Since the eigenvalues $\lambda_n$ of the Laplace operator
form an increasing divergent sequence, we can choose $N_B\in \NN$ such that
\begin{equation}
\frac{\lambda_{N_B}}{16} \geq B.
\end{equation}
By the definition of $\mathcal{K}_1$, we have that
\begin{equation}\label{eq:pic1}
\| \widetilde{\Pi}_cf_0\|^2_{L^2}\geq \frac34\|f_0\|^2_{L^2},
\end{equation}
or, equivalently,
\begin{equation}\label{eq:pic2}
\| \widetilde{\Pi}_pf_0\|^2_{L^2}\leq \frac14\|f_0\|^2_{L^2}.
\end{equation}
Also, exploiting Lemma \ref{lem:CKRZ1}, fix $T_c=T_c(B,\mathcal{K})$ such that
\begin{equation}\label{eq:fix}
\frac1T\int_0^T \|P_{\leq N_B}\widetilde{S}(t)\widetilde{\Pi}_cf_0\|^2_{L^2}\dd t\leq \frac18 \|f_0\|_{L^2}^2, 
\end{equation}
for every $T \geq T_c$ and every $f_0 \in \mathcal{K}$.
Since $\widetilde{S}(t)$ is unitary, for each $t\geq 0$ we have
\begin{align}
\|(I-P_{\leq N_B})\widetilde{S}(t)f_0\|^2_{L^2}&\geq \frac12 \|(I-P_{\leq N_B})\widetilde{S}(t)\widetilde{\Pi}_cf_0\|^2_{L^2}-\|(I-P_{\leq N_B})\widetilde{S}(t)\widetilde{\Pi}_pf_0\|^2_{L^2}\\
&\geq \frac12 \|\widetilde{S}(t)\widetilde{\Pi}_cf_0\|^2_{L^2}-\frac12 \|P_{\leq N_B}\widetilde{S}(t)\widetilde{\Pi}_cf_0\|^2_{L^2} -\|\widetilde{S}(t)\widetilde{\Pi}_pf_0\|^2_{L^2}\\
&=\frac12 \|\widetilde{\Pi}_cf_0\|^2_{L^2}-\frac12 \|P_{\leq N_B}\widetilde{S}(t)\widetilde{\Pi}_cf_0\|^2_{L^2} -
\|\widetilde{\Pi}_pf_0\|^2_{L^2},
\end{align}
and hence with \eqref{eq:pic1}, \eqref{eq:pic2} we conclude
\begin{align}
\|(I-P_{\leq N_B})\widetilde{S}(t)f_0\|^2_{L^2}
&\geq \frac38 \|f_0\|^2_{L^2}-\frac12 \|P_{\leq N_B}\widetilde{S}(t)\widetilde{\Pi}_cf_0\|^2_{L^2} -\frac14\|f_0\|^2_{L^2}\\
&=\frac18 \|f_0\|^2_{L^2}-\frac12 \|P_{\leq N_B}\widetilde{S}(t)\widetilde{\Pi}_cf_0\|^2_{L^2}.
\end{align}
From the above inequality and \eqref{eq:fix}, we then learn that 
\begin{equation}
\frac1T\int_0^T\|(I-P_{\leq N_B})\widetilde{S}(t)f_0\|^2_{L^2}\dd t\geq \frac{1}{16} \|f_0\|^2_{L^2}, \qquad \forall T\geq T_c.
\end{equation}
Therefore, for each $T\geq T_c$ we have
\begin{align}
\frac1T\int_0^T\|\widetilde{S}(t)f_0\|^2_{H^1}\dd t 
&\geq \frac1T\int_0^T\|(I-P_{\leq N_B})\widetilde{S}(t)f_0\|^2_{H^1}\dd t \\
&\geq \frac{\lambda_{N_B}}{T}\int_0^T\|(I-P_{\leq N_B})\widetilde{S}(t)f_0\|^2_{L^2}\dd t \\
&\geq \frac{\lambda_{N_B}}{16} \|f_0\|^2_{L^2}\\
&\geq B \|f_0\|^2_{L^2},
\end{align}
for all $T \geq T_c$ and every $f_0 \in \mathcal{K} \backslash \mathcal{K}_1$.
Hence, \eqref{eq:growth1} is proven in the second case as well, and the proof is concluded.
\end{proof}

Recall that a Borel probability measure $\mu_0\in \MM(L^2)$
 is called an \emph{invariant measure} for $S(t)$ 
if, for every $t\in \RR$,
\begin{equation}\label{eq:inverse}
\mu_0(A)=\mu_0(S(t)A), \qquad \forall A\in \BB(L^2).
\end{equation}
In an equivalent way, a measure $\mu_0$ is invariant for $S(t)$ if
\begin{equation}\label{eq:inverse1}
\int_{L^2} \varphi(\zeta)\dd \mu_0(\zeta)=\int_{L^2}\varphi(S(t)\zeta)\dd \mu_0(\zeta),
\end{equation}
for every $t \in\RR$ and every bounded real-valued continuous function $\varphi$ on $L^2$. 
Notice that there is no need to take the inverse image of $S(t)$ in \eqref{eq:inverse},
since we are dealing with a \emph{group} of operators. The support of $\mu_0$,
denoted by $\spt(\mu_0)$, is the intersection of all closed sets with measure one according to $\mu_0$.

As a consequence of Theorem \ref{thm:rigidity}, we infer the following information about
the support of a certain class of invariant measure $\mu_0$ of $S(t)$. 

\begin{corollary}\label{cor:sptinviscid}
Let $\mu_0 \in \MM(L^2)$ be an invariant measure for $S(t)$ such that
\begin{equation}\label{eq:H1spt}
\int_{L^2} \|\zeta\|_{H^1}^2\dd \mu_0(\zeta)<\infty.
\end{equation}
Then $\mu_0(H^1\cap E)=1$. In particular, $\spt(\mu_0)\subset E$.
\end{corollary}
In the proof of this result, we will make use of the invariance property \eqref{eq:inverse1}
for the function $\varphi(\cdot)=\|\cdot\|^2_{H^1}$, which is only assumed to be in $L^1(\mu_0)$.
 To justify this, 
for each $n\in \NN$, we truncate the $H^1$ norm on the Fourier side by defining the
sequence of functions
$\varphi_n(\zeta)=
\max\{\varphi(P_{\leq n}\zeta),n\}$. 
Then $\{\varphi_n\}_{n\in\NN}\subset C_b(L^2)$ and by \eqref{eq:inverse1}  we infer that
\begin{equation}
\int_{L^2}\varphi_n(S(t)\zeta)\dd\mu_0(\zeta)=\int_{L^2}\varphi_n(\zeta)\dd\mu_0(\zeta), \quad \forall n\in\NN.
\end{equation}
We can then take $n\to \infty$ and use the monotone convergence theorem to obtain
\begin{equation}\label{eq:H1inv}
\int_{L^2}\|S(t)\zeta\|_{H^1}^2\dd\mu_0(\zeta)=\int_{L^2}\|\zeta\|^2_{H^1}\dd\mu_0(\zeta), \qquad \forall t\geq 0.
\end{equation}

\begin{proof}[Proof of Corollary \ref{cor:sptinviscid}]
Firstly notice that a straightforward application of Chebyshev's inequality together with
assumption \eqref{eq:H1spt} implies that $\mu_0(H^1)=1$.
Fix $c_1>0$ such that
\begin{equation}
\int_{L^2} \|\zeta\|_{H^1}^2\dd \mu_0(\zeta)\leq c_1.
\end{equation}
We need to show that $\mu_0 (E)=1$. 
Suppose not. Then $\mu_0 (E^c \cap H^1) > 0$.
First, consider the case that $\mu_0(E^{\perp} \cap H^1 \setminus \set{0}) > 0$ (note that this holds in the case $E = \set{0}$).  
Then, by the inner regularity of the measure $\mu_0$, 
we can deduce the existence of a compact set 
$\mathcal{K}\subset (E^\perp\cap H^1)\setminus \set{0}$
and an $\eps > 0$  such that
\begin{equation}
\mu_0(\mathcal{K})>\eps. 
\end{equation}
Moreover, we may restrict ourselves to $\mathcal{K}$ such that $\sup_{f \in \mathcal{K}}\norm{f}_{H^1} < \infty$. 
Fix a positive constant $M$ such that
 \begin{equation}
 M\geq \frac{2c_1}{\eps}.
 \end{equation}
By Theorem \ref{thm:rigidity}, we can find $T_M>0$ large enough so that 
\begin{equation}
\inf_{f_0\in \mathcal{K}}\frac1{T_M}\int_0^{T_M} \|S(t)f_0\|_{H^1}^2 \dd t \geq M, 
\end{equation} 
Hence,
\begin{equation}\label{eq:kappa}
\eps < \mu_0\left(\mathcal{K}\right)\leq \mu_0\left(f_0\in H^1: \frac1{T_M}\int_0^{T_M}\|S(t)f_0\|^2_{H^1}\dd t\geq M\right).
\end{equation}
We will see that this is sufficient to rule out the existence of $\mathcal{K}$. 

Second, consider the case that $E$ is non-trivial and $\mu_0(E^{\perp} \cap H^1 \setminus \set{0}) = 0$.
Let $\set{\phi_j}_{j=1}^{\bar{N}}$ (with $\bar{N} \leq \infty$) be an orthonormal basis for $E$ consisting of $H^1$ eigenvalues of $L$. 
By the continuity of $\mu_0$ with respect to decreasing sequences of sets, we have
\begin{align*}
\lim_{N \rightarrow \infty} \mu_0\left( \left(\textup{span}\left(\set{\phi_j}_{j=N}^{\bar{N}}\right) \oplus E^{\perp}\right) \cap E^c \right) = 0, 
\end{align*}
and hence for $N < \bar{N}$ sufficiently large, there holds 
\begin{align*}
\mu_0\left(\left(\textup{span}\left(\set{\phi_j}_{j=1}^{N}\right) \oplus E^{\perp}\right) \cap E^c \right) > 0. 
\end{align*} 
By inner regularity, there exists a compact set $\mathcal{K} \subset \left(\textup{span}(\phi_1,...,\phi_N) \oplus E^{\perp}\right) \cap E^c$ and an $\eps > 0$ such that
\begin{align*}
\mu_0(\mathcal{K}) > \eps.
\end{align*}
As above, we may further restrict ourselves to $\mathcal{K}$ such that $\sup_{f \in \mathcal{K}}\norm{f}_{H^1} < \infty$. 
Due to the fact that $N < \infty$, it follows that $\Pi_e$ maps $\mathcal{K}$ into $H^1$, that is, we have $\Pi_e:\textup{span}(\phi_1,...,\phi_N) \oplus E^{\perp} \rightarrow E \cap H^1$ as a bounded linear operator.  
Hence, there is a constant $C_N$ (depending only on $N$) such that for an arbitrary $f \in \mathcal{K}$,  
\begin{align*}
\norm{S(t)f_0}_{H^1} \geq \norm{S(t)(I - \Pi_e)f_0}_{H^1} - \norm{S(t)\Pi_e f_0}_{H^1} \geq \norm{S(t)(I - \Pi_e)f_0}_{H^1} - C_N \sup_{f_0 \in \mathcal{K}}\norm{f_0}_{H^1}. 
\end{align*}
By Theorem \ref{thm:rigidity}, for any $M'$, we can find $T_{M'}>0$ large enough so that 
\begin{equation}
\inf_{f_0\in (I-\Pi_e)\mathcal{K}}\frac1{T_{M'}}\int_0^{T_{M'}} \|S(t)f_0\|_{H^1}^2 \dd t \geq M', 
\end{equation} 
and hence, by choosing $M'$ sufficiently large relative to $C_N \sup_{f \in \mathcal{K}}\norm{f_0}_{H^1}$, and possibly increasing $T_M$, 
we have that
\begin{align}
\eps < \mu_0\left(\mathcal{K}\right)\leq \mu_0\left(f_0\in H^1: \frac1{T_{M}}\int_0^{T_{M}}\|S(t)f_0\|^2_{H^1}\dd t\geq M\right).\label{ineq:kappa2}
\end{align} 
Since $\mu_0$ is invariant and supported on $H^1$, we use Fubini's theorem and \eqref{eq:H1inv} to obtain
\begin{align}
\int_{L^2}\frac1{T_M} \int_{0}^{T_M}\|S(t)\zeta\|^2_{H^1}\dd t\dd\mu_0(\zeta)
&=\frac1{T_M} \int_{0}^{T_M}\int_{L^2}\|S(t)\zeta\|^2_{H^1}\dd\mu_0(\zeta)\dd t\\
&=\frac1{T_M} \int_{0}^{T_M}\int_{L^2}\|\zeta\|^2_{H^1}\dd\mu_0(\zeta)\dd t\\
&=\int_{L^2}\|\zeta\|^2_{H^1}\dd\mu_0(\zeta)\leq c_1.
\end{align}
To conclude, Chebyshev's inequality, \eqref{eq:kappa} or \eqref{ineq:kappa2}, and our choice of $M$ imply that
\begin{equation}\label{eq:compu}
\begin{aligned}
\eps&<\mu_0\left(\mathcal{K}\right)\leq\mu_0\left(f_0\in H^1: \frac1{T_M}\int_0^{T_M}\|S(t)f_0\|^2_{H^1}\dd t\geq M\right)\\
&\leq \frac{1}{M}\int_{L^2}\frac1{T_M} \int_{0}^{T_M}\|S(t)\zeta\|^2_{H^1}\dd t\,\dd\mu_0(\zeta)
\leq \frac{c_1}{M}\leq \frac{\eps}{2},
\end{aligned}
\end{equation}
a contradiction. This finishes the proof.
\end{proof}

\begin{remark}\label{rmk:nonlin}
Corollary \ref{cor:sptinviscid} essentially exploits the norm growth \eqref{eq:growth} available for the linear
semigroup $S(t)$ for data in any compact set $\mathcal{K}\subset E^\perp$. In principle, linearity is not
needed as long as \eqref{eq:growth} holds. To make this precise, let $R(t):L^2\to L^2$ be a 
(possibly nonlinear) semigroup, and assume that there exists an invariant measure $\mu_0 \in \MM(L^2)$  
for $R(t)$ such that
\begin{equation}
\int_{L^2} \|\zeta\|_{H^1}^2\dd \mu_0(\zeta)<\infty.
\end{equation}
If there exists a compact set $\mathcal{K}$ such that 
$$
\lim_{T\to\infty} \inf_{f_0\in \mathcal{K}}\frac1T\int_0^T \|R(t)f_0\|_{H^1}^2\dd t=\infty,
$$
then $\mu_0(\mathcal{K})=0$. Indeed, if $\mu_0(\mathcal{K})>\eps>0$, a computation analogous to \eqref{eq:compu}
produces a contradiction.
\end{remark}

\subsection{Extensions to an abstract setting}
For simplicity of the presentation, we proved Theorem \ref{thm:rigidity} and Corollary \ref{cor:sptinviscid}
in the  case of the space $L^2$, the operator $\uu\cdot \grad$, and the scale of standard 
Sobolev spaces generated by the Laplace operator. Since Lemmas \ref{lem:CKRZ1}-\ref{lem:CKRZ2}
are valid in a more general setting (see \cite{CKRZ08}), Theorem \ref{thm:rigidity} also holds in greater generality.
We here state the more general context in
which the results of the previous section hold.

Let $(H,\|\cdot\|_H)$ be a Hilbert space, and let $A$ be a strictly
positive self-adjoint linear operator
$$
A:D(A)\subset H\to H,
$$
such that $D(A)$ is compactly embedded in $H$. 
From classical spectral theory \cite{Y80}, we have that $A$ possesses a strictly positive sequence of eigenvalues $\{\lambda_k\}_{k \in \NN}$ such that
\begin{align*}
\begin{cases}
0<\lambda_1\leq \lambda_2 \leq \ldots,\\
\lambda_k\to \infty \quad\text{for} \quad k \to \infty,
\end{cases}
\end{align*}
and associated eigenvectors $\{e_k\}_{k\in\NN}$ which form an orthonormal basis for $H$.
Using the powers of $A$, we can define the  Hilbert space
$$
H^1=D(A^{1/2}),\qquad \|\varphi\|_{H^1}=\|A^{1/2}\varphi\|  .
$$
In particular
$$
\lambda_1\|\varphi\|^2_{H} \leq \|\varphi\|^2_{H^1}.
$$
Let $L$ be a self-adjoint linear operator such that there exists a $c >0$ and $B\in L^2_{loc}(0,\infty)$ such that 
for any $\varphi\in H^1$ and $t>0$,
$$
\| L\varphi\|_H\leq c\|\varphi\|_{H^1}, \qquad \| \e^{iLt}\varphi\|_{H^1}\leq B(t) \|\varphi\|_{H^1}. 
$$
Here $\e^{iLt}$ is the unitary group on $H$ generated by the ordinary differential
equation 
$$
f'-iLf=0, \qquad f(0)=f_0\in H.
$$
Defining the subspace $E$ spanned by the $H^1$ eigenfunctions of $L$ as in \eqref{eq:smootheigen}, namely
\begin{equation}
E= \overline{\mbox{span}\big\{\varphi\in H^1: L\varphi= \lambda \varphi, \ \lambda\in \RR \big\}}^{L^2},
\end{equation}
the abstract version of Theorem \ref{thm:rigidity2} reads as follows.

\begin{theorem}\label{thm:rigidity2}
Let $\mathcal{K}\subset H^1\cap E^\perp$ be a nonempty compact set in $H$ such that 
$0\notin \mathcal{K}$. Then
\begin{equation}
\lim_{T\to\infty} \inf_{f_0\in \mathcal{K}}\frac1T\int_0^T \|\e^{iLt}f_0\|_{H^1}^2\dd t=\infty.
\end{equation}
\end{theorem}
We then have an analogue of Corollary \ref{cor:sptinviscid}.
\begin{corollary}\label{cor:sptinviscid2}
Let $\mu_0 \in \MM(H)$ be an invariant measure for $\e^{iLt}$ such that
\begin{equation}
\int_{H} \|\zeta\|_{H^1}^2\dd \mu_0(\zeta)<\infty.
\end{equation}
Then $\mu_0(H^1\cap E)=1$. In particular, $\spt(\mu_0)\subset E$.
\end{corollary}
Theorem \ref{thm:rigidity2} and Corollary \ref{cor:sptinviscid2} provide a general result that applies to the following cases, some of which may be of wider interest:
\begin{itemize} 
\item dynamical systems posed on (finite dimensional) Riemannian manifolds without boundaries, indeed, the choice of $\Torus^d$ in \eqref{eq:eq1} was arbitrary and simply for clarity of exposition; 
\item small noise limits of inhomogeneous diffusion problems, for example, 
\begin{align*}
\dd f + \left(\uu \cdot \grad f - \nu\grad \cdot (A(x) \grad f) \right) \dd t = \sqrt{\nu} \, \Psi\, \dd W_t, \qquad f(0)=f_0,
\end{align*}
where $A(x)$ is smooth, symmetric, and uniformly positive definite;    
\item fractional order dissipation; this is discussed further in Section \ref{sec:GenSob} below.   
\end{itemize} 

\subsection{Generalizations to different Sobolev norms} \label{sec:GenSob}
The concrete case discussed in Section \ref{sub:invmsr} corresponds to
$$
H=\left\{\varphi\in L^2: \int_{\TT^d}\varphi(x)\dd x=0\right\}, \qquad A=-\Delta, \qquad L=i\uu\cdot \grad.
$$
By modifying the above setting to 
$$
H=\left\{\varphi\in L^2: \int_{\TT^d}\varphi(x)\dd x=0\right\}, \qquad A=(-\Delta)^{2s}, \qquad L=i\uu\cdot \grad,
$$
for $s>0$, it is easily seen that more general versions of the  previous results hold.
Thanks to Theorem \ref{thm:rigidity2}, it is clear that the classical Sobolev space $H^1$
in Theorem \ref{thm:rigidity}
does not play a specific role other than being the domain of the square root of the Laplace operator.
Correspondingly, for fractional dissipation we may define
\begin{equation}
E= \overline{\mbox{span}\big\{\varphi\in H^s: \uu\cdot \nabla\varphi=i \lambda \varphi, \ \lambda\in \RR \big\}}^{L^2},
\end{equation}
and we then have the following.
\begin{theorem}\label{thm:rigidity3}
Let $s>0$ and $\mathcal{K}\subset H^s\cap E^\perp$ be a nonempty compact set  in $L^2$ such that 
$0\notin \mathcal{K}$. Then
\begin{equation}
\lim_{T\to\infty} \inf_{f_0\in \mathcal{K}}\frac1T\int_0^T \|S(t)f_0\|_{H^s}^2\dd t=\infty.
\end{equation}
\end{theorem}
The importance of the above observation relies on the fact that invariant measures to the deterministic
inviscid equation \eqref{eq:inviscid} will be constructed via a particular vanishing viscosity limit 
of dissipative stochastic flows. If the dissipation is generated by the Laplacian, 
viscous invariant measures will satisfy a bound analogous to \eqref{eq:H1spt} which turns out to be 
stable under the limit procedure. However, if for example the dissipation is given by a fractional 
power of the Laplacian, only a weaker Sobolev norm will be preserved (or a stronger Sobolev norm if one takes $s > 1$). We therefore
state the more general version of Corollary \ref{cor:sptinviscid} hereafter. 

\begin{corollary}\label{cor:sptinviscid3}
Let $s>0$ and $\mu_0 \in \MM(L^2)$ be an invariant measure for $S(t)$ such that
\begin{equation}
\int_{L^2} \|\zeta\|_{H^s}^2\dd \mu_0(\zeta)<\infty.
\end{equation}
Then $\mu_0(H^s\cap E)=1$. In particular, $\spt(\mu_0)\subset E$.
\end{corollary}

\section{Inviscid deterministic limit of viscous stochastic measures}\label{sec:section3}

An interesting class of invariant measures for
infinite dimensional Hamiltonian systems such as the 2D Euler and KdV equations may be obtained from a
viscous-stochastic perturbation where the noise and dissipation terms are carefully balanced.
While the measures obtained from such a procedure have been studied extensively in a series of recent works, see e.g.
\cites{K04, K10, KS04, KS12, MP14}, their structure remains poorly understood.  
Here and in the sequel Section~\ref{sec:examples} we considered this limit in a linear setting and show that the results in the
previous section can be used to obtain significant information about the structure 
of these limiting measures. 

As in the previous Section \ref{sec:inviscid}, we fix a divergence free Lipschitz flow $\uu$ and
consider the corresponding stochastically forced, linear system 
\begin{equation}\label{eq:SPDEviscous}
\dd f + \left(\uu \cdot \grad f - \nu \Delta f\right) \dd t = \sqrt{\nu} \, \Psi\, \dd W_t
=\sqrt{\nu}\sum_{k\in \NN} \psi_k e_k \,\dd W^k_t, \qquad f(0)=f_0,
\end{equation} 
evolving on $\TT^d$ where $\nu\in (0,1]$ is a diffusivity parameter and $\psi_k\geq 0$ are coefficients satisfying 
$$
  \|\Psi\|^2 =\sum_{k\in\NN} |\psi_k|^2<\infty.
$$
The sequence $W_t =\{W^k_t\}_{k\in\NN}$ consists of independent copies of the standard one-dimensional
Wiener process 
(Brownian motion).  As such, for each $k$, $\dd W^k_t$ is formally a 
white noise which, in particular, is stationary in time.

Having fixed a stochastic basis $\mathcal{S}=(\Omega, \mathcal{F}, \{\mathcal{F}_t\}_{t\geq 0}, \PP, W_t)$
and an $\mathcal{F}_0$ measurable initial datum $f_0\in L^2$ 
the existence of a unique weak solution to \eqref{eq:SPDEviscous} can be deduced by classical stochastic PDEs methods;
see \cite{DZ92}. 
More precisely for each $\nu\in (0,1]$, there exists a unique $L^2$-valued random process $\{f^\nu(t)\}_{t \geq 0}$
with $f^\nu(0)=f_0$ almost surely and such that:
\begin{enumerate}
	\item The process $f^\nu(t)$ is $\mathcal{F}_t$-adapted  and
	\begin{equation}
	f^\nu\in C(\RR^+; L^2)\cap L^2_{loc}(\RR^+; H^1). 
	\end{equation}
	almost surely.
	\item Equation \eqref{eq:SPDEviscous} is satisfied in the time integrated sense
	\begin{equation}\label{eq:solsense}
	f^\nu(t)+\int_0^t \big[\uu \cdot \grad f^\nu(s) - \nu \Delta f^\nu(s)\big]\dd s= f^\nu(0)+ \sqrt{\nu}\,\Psi\,W_t ,
	\end{equation}
	with probability 1 for each $t \geq 0$.  Here the equality holds in the space $H^{-1}$.
\end{enumerate} 
When $f_0 \in L^2(\Omega, L^2)$ we have $f \in L^2(\Omega; L^\infty_{loc}(\RR^+; L^2)\cap L^2_{loc}(\RR^+; H^1))$
and these solutions of \eqref{eq:SPDEviscous} are easily seen to satisfy the energy balance equation
\begin{equation}\label{eq:energy}
	\EE \|f^\nu(t)\|^2_{L^2}+2\nu \EE\int_\tau^t\|f^\nu(s)\|^2_{H^1}\dd s=
	\EE \|f^\nu(\tau)\|^2_{L^2}+\nu\|\Psi\|^2 (t-\tau)
\end{equation}
which holds for any $t> \tau\geq 0$.    Moreover using exponential martingale estimates
one has that
\begin{align}
  \mathbb{P} \left( \sup_{t \geq 0} \left( \|f^\nu(t)\|^2_{L^2}+\nu \int_0^t\|f^\nu(s)\|^2_{H^1}\dd s - t \nu\| \Psi\|^2 - \|f_0\|^2 \right)  > K \right) \leq 
  \e^{- \frac{\lambda_1 }{2\| \Psi\|^2}K}
\end{align}
for every $K > 0$, which yields additional exponential moments (see e.g. \cite{KS12}).

It is worth emphasizing that, in contrast to 
active scalar systems like the stochastic Navier-Stokes, we can identify the distribution of solutions of \eqref{eq:SPDEviscous}.
For this, consider the linear deterministic counterpart of \eqref{eq:SPDEviscous}
\begin{equation} \label{eq:viscous}
\de_t f + \uu\cdot \grad f-\nu\Delta f=0,\qquad f(0)=f_0.
\end{equation}
For any $\nu>0$, the associated semigroup generated by \eqref{eq:viscous} will be denoted by
$$
S_\nu(t):L^2\to L^2.
$$
Note that the adjoint $S_\nu(t)^\ast$ is the solution operator associated with 
\begin{equation} \label{eq:viscousadjoint}
\de_t f - \uu\cdot \grad f-\nu\Delta f=0,\qquad f(0)=f_0.
\end{equation} 
Given any deterministic $f_0 \in L^2$ we have that 
$f^\nu(t)$ is Gaussian with mean $S_\nu(t)f_0$
and variance given as 
\begin{align*}
Q_\nu(t) = \nu\int_0^t S_\nu(s)\Psi \Psi^\ast S_\nu(s)^\ast \dd s. 
\end{align*}

\subsection{The Markovian framework and stationary statistical solutions}\label{sub:Markov}
Associated to \eqref{eq:SPDEviscous} is the so-called \emph{Markov semigroup} $\{\MA_t\}_{t\geq 0}$,
defined on the space $M_b(L^2)$  as 
\begin{equation}
\MA_t \varphi(f_0)= \EE\varphi(f^\nu(t,f_0)), \qquad \varphi\in M_b(L^2),\ t\geq 0.
\end{equation}
Here, we stress the dependence on the initial datum by writing $f^\nu(t,f_0)$ for the solution
to \eqref{eq:SPDEviscous} emanating from $f_0$. Since $f^\nu(t,f_0)$ depends continuously on $f_0$,
it follows that $\{\MA_t\}_{t\geq 0}$ is \emph{Feller}, namely, it also maps  $C_b(L^2)$ to itself. 

For each $\nu>0$, the classical Krylov-Bogolyubov procedure establishes
the existence of an \emph{invariant} measure $\mu_\nu\in \MM(L^2)$ for \eqref{eq:SPDEviscous}, that is an element such that
\begin{equation}
\int_{L^2}\MA_t\varphi(\zeta)\dd\mu_\nu(\zeta)=\int_{L^2}\varphi(\zeta)\dd\mu_\nu(\zeta), \qquad \forall t\geq 0.
\end{equation}
Such measures correspond to statistically invariant states of \eqref{eq:SPDEviscous}.
Unlike in the nonlinear setting (see e.g. \cites{DZ96,KS12} and the references therein), the uniqueness of $\mu_\nu$ is not 
an issue here. Indeed, as $S_\nu(t)$ is an exponentially stable dynamical system, its only invariant
measure is the Dirac mass centered at zero. Therefore,
\cite{DZ96}*{Theorem 6.2.1} provides
a precise characterization of $\mu_\nu$. Specifically,
$$
\mu_\nu =\mathcal{N}(0,Q_\nu),
$$
a Gaussian centered at 0 with covariance operator given by
\begin{equation}
Q_\nu = \nu\int_0^\infty S_\nu(t)\Psi \Psi^\ast S_\nu(t)^\ast \dd t. 
\end{equation}
We denote by $f_S^\nu(t)$ a statistically stationary solution associated to $\mu_\nu$,
for which
\begin{equation}
\PP(f_S^\nu(t)\in A)=\mu_\nu(A), \qquad \forall A\in \BB(L^2),\, t\geq 0.
\end{equation}
In particular, it follows from the energy equation \eqref{eq:energy} that any statistically
stationary solution obeys the stronger balance
\begin{equation}\label{eq:H1balance}
\EE\|f_S^\nu (t)\|^2_{H^1}=\int_\Omega \|f_S^\nu (t)\|^2_{H^1}\dd \PP= 
\int_{H^1} \|\zeta\|^2_{H^1}\dd \mu_\nu(\zeta)=\frac12\| \Psi\|^2, \qquad \forall t\geq 0.
\end{equation}
Similarly to \cite{KS12} (and see also e.g. \cites{FG95, DGT11}) we have the following further $\nu$-independent bounds.

\begin{lemma}\label{lem:stat}
Let $f_S^\nu$ be a statistically stationary solution associated to the invariant measure $\mu_\nu$. 
For each $T>0$, define the trajectory space
\begin{equation}\label{eq:stat2}
\YY_T=L^2(I_T;H^1)\cap(H^1(I_T;H^{-1})+ W^{\alpha,4}(I_T;L^2))
\end{equation}
where $I_T=[0,T]$ and $\alpha\in (1/4,1/2)$. Then,
\begin{equation}\label{eq:stat1}
\EE\|f_S^\nu\|^2_{\YY_T}\leq c_0,
\end{equation}
where $c_0=c_0(\|\uu\|_{L^\infty},\alpha, T, \| \Psi\|^2)>0$ is independent of $\nu \in (0,1]$.
\end{lemma}

\begin{proof}
The first part of the bound in $L^2(I_T;H^1)$ follows directly from \eqref{eq:H1balance}.
For the second bound in $H^1(I_T;H^{-1})+ W^{\alpha,4}(I_T;L^2)$ we split \eqref{eq:solsense} as, 
\begin{equation} 
f_S^\nu(t)=g(t)+ \sqrt{\nu} \Psi W_t,
\end{equation}
where
\begin{equation}
g^\nu(t)=-\int_0^t \big[\uu \cdot \grad f_S^\nu(s) - \nu \Delta f_S^\nu(s)\big]\dd s+ f^\nu(0).
\end{equation}
Observe that
\begin{align} 
\|\de_t g^\nu(t)\|_{H^{-1}}&\leq \left(\|\uu\|_{L^\infty}\|f_S^\nu(t)\|_{L^2}+\nu\|f_S^\nu(t)\|_{H^1}\right)\\
&\leq (1+\|\uu\|_{L^\infty})\|f_S^\nu(t)\|_{H^1}.
\end{align}
Similarly 
\begin{align*}
 \|g^\nu(t)\|^2_{H^{-1}} \leq c(1+\|\uu\|_{L^\infty})\int_0^T\|f_S^\nu(t)\|_{H^1}^2 dt +  c\|f^\nu_S(0)\|_{H^1}^2.
\end{align*}
As a consequence, making another use of  \eqref{eq:H1balance} we conclude that
\begin{equation}
\EE\int_0^T ( \|g^\nu(t)\|^2_{H^{-1}}  + \|\de_t g^\nu(t)\|^2_{H^{-1}})\dd t\leq c(1+\|\uu\|^2_{L^\infty})\| \Psi\|^2,
   \label{eq:div:p:e:1}
\end{equation}
for a constant $c >0$ independent of $\nu \in (0,1]$.
Since $\Psi W_t - \Psi W_s \sim \mathcal{N}(0, \Psi (t-s))$ for any $t > s \geq 0$ we have that
$$
\EE \| \Psi W_t - \Psi W_s\|^4 \leq c\| \Psi\|^4(t-s)^2
$$ 
which yields the estimate
\begin{equation}
\EE\int_0^T\|\Psi W_t\|_{L^2}^4\dd t+\EE\int_0^T\int_0^T\frac{\| \Psi W_t- \Psi W_s\|_{L^2}^4}{|t-s|^{1+4\alpha}}\,\dd t\, \dd s\leq c(T) \| \Psi\|^4.
   \label{eq:div:p:e:2}
\end{equation}
Combining \eqref{eq:div:p:e:1} and \eqref{eq:div:p:e:2} now gives the second $\nu$-independent bound concluding the proof.
\end{proof}

\subsection{The inviscid limit}
Thanks to the compactness of the embedding of $H^1$ into $L^2$ and \eqref{eq:H1balance}, 
the collection $\{\mu_\nu\}_{\nu\in(0,1]}$ is easily seen to be tight.
As such one can extract weakly convergent subsequences and we will refer to any
limiting probability measure, denoted by $\mu_0$, as a \emph{Kuksin} measure.  As mentioned above, such measures have been extensively studied 
in an analogous nonlinear setting \cites{K04, K10, KS04, KS12, GSV13, MP14}. 
Let us now recall some properties of $\mu_0$ which may be obtained in a similar manner to these works. 

We begin by observing that invariance is preserved in this inviscid limit
\begin{proposition}\label{prop:inviscid}
The measure $\mu_0\in \MM(L^2)$ is invariant under the group $\{S(t)\}_{t \in \RR}$ defined by \eqref{eq:inviscid}, namely
\begin{equation}
\mu_0(A)=\mu_0(S(t)A), \qquad \forall A\in \BB(L^2), \ t\in \RR.
\end{equation}
\end{proposition}
As discussed in Section \ref{sec:inviscid},
the inviscid problem \eqref{eq:inviscid} is well-posed for initial data in $L^2$; let $\XX$ be
the set of all solutions to \eqref{eq:inviscid}. As we have seen,
\begin{equation}
\XX\subset C_b(\RR;L^2)\cap W^{1,\infty}(\RR;H^{-1}).
\end{equation}
Define $K_0:\XX\to L^2$ by $K_0 \varphi=\varphi(0)$. From uniqueness, it follows that $K_0$ is one-to-one, while
existence for any arbitrary initial datum $f_0\in L^2$ shows that $K_0$ is onto, hence invertible.
As a consequence,
from any Borel probability measure $\mu$ on $L^2$ it is possible to define a \emph{lifted} probability measure $\mmu$ 
on $\XX$ via
$$
\mmu(A)=\mu(K_0 A), \qquad A\in  \BB(\XX).
$$
The proof of Proposition \ref{prop:inviscid} is very similar to that in \cite{KS12}*{Theorem 5.2.2}, 
and see also \cite{FG95}.  We therefore omit the details of the following steps, based on compactness arguments and
probabilistic methods.

\begin{itemize}
	\item To the sequence $\{\mu_\nu\}_{\nu\in(0,1]}\subset \MM(L^2)$, we associate the sequence
	of lifted measures on trajectories $\{\mmu_\nu\}_{\nu\in(0,1]}\subset \MM(\YY_T)$. The latter is tight in  
	$C(\RR^+;H^{-\eps})\cap L^2_{loc}(\RR^+;H^{1-\eps})$, for any $\eps>0$, 
	thanks to  Lemma \ref{lem:stat}, hence limit points $\mmu_0$ exist.
	
	\item In view of \eqref{eq:stat1}, $\mmu_0(L^2(I_T;H^1))=1$, and $\mmu_0$ is in fact the lifting of the
	measure $\mu_0$. Moreover, $\mmu_0$ is the law of a stationary process $f_S$ whose 
	trajectories solve the inviscid equation \eqref{eq:inviscid}, at least up to a set of measure zero.  This 
	implies the $\mu_0$ is invariant under $S(t)$.
\end{itemize}

Let us next highlight some further properties of $\mu_0$ and its associated statistically
stationary solutions $f_S$.

\begin{lemma}\label{lem:invstat1}
Let $f_S$ be a  statistically stationary solution of \eqref{eq:inviscid} associated to a Kuksin measure $\mu_0$. Then almost every realization
of $f_S$ belongs to the space $\XX$.
\end{lemma}

\begin{proof}
By construction, any statistically stationary solution $f_S$ is a solution to the 
inviscid problem \eqref{eq:inviscid} and a limit point of a subsequence
of statistically stationary solutions $f_S^\nu$ associated to $\mu_{\nu}$. Almost surely
and for every $T>0$, $f_S^\nu$  belongs to the trajectory space $\YY_T$ (see \eqref{eq:stat2}),
with
\begin{equation}
\EE\|f_S^\nu\|_{\YY_T}^2\leq c_0.
\end{equation}
In turn, a lower semicontinuity argument implies that the same holds for weak subsequential limits, namely
\begin{equation}
\EE\|f_S\|_{\YY_T}^2\leq c_0.
\end{equation}
Since the space $\YY_T$ (see e.g. \cites{K02,L69, LM72}) is continuously embedded in $C(I_T;L^2)$, we 
infer that almost surely 
\begin{equation}
f_S\in C(\RR^+;L^2).
\end{equation}
Moreover, $f_S$ can be extended backward in time due to time-reversibility of the inviscid equation. 
As a consequence, any statistically stationary solution $f_S$ to \eqref{eq:inviscid} is global in time, belongs to $\XX$, and satisfies
the global estimate \eqref{eq:globalest}.
\end{proof}

Besides the above features, the measure $\mu_0$ possesses an additional property that is
essential to our analysis.

\begin{lemma}\label{lem:H1bdd}
Let $\mu_0$ be a Kuskin measure and let $f_S\in \XX$ be a statistically stationary solution
to \eqref{eq:inviscid} associated to $\mu_0$. Then 
\begin{equation}\label{eq:H1bdd}
\int_{L^2}\|\zeta\|^2_{H^1}\dd\mu_0(\zeta)=\EE\, \frac1T\int_0^T \|f_S(t)\|^2_{H^1}\dd t\leq \frac12 \| \Psi\|^2,
\end{equation}
for every $T>0$. 
\end{lemma}

\begin{proof}
The equality in \eqref{eq:H1bdd} is simply a consequence of the fact that $f_S$ is a statistically stationary
solution associated to $\mu_0$. Now, for each $\nu>0$ and thanks to stationarity, \eqref{eq:energy} implies that
\begin{equation}
\EE\, \frac1T\int_0^T \|f^\nu_S(t)\|^2_{H^1}\dd t= \frac12 \| \Psi\|^2, \qquad \forall T>0.
\end{equation}
The uniformity with respect to $\nu>0$ of the above estimate together with weak compactness and lower
semicontinuity implies that
\begin{equation}
\EE\, \frac1T\int_0^T \|f_S(t)\|^2_{H^1}\dd t\leq \frac12 \| \Psi\|^2, \qquad \forall T>0.
\end{equation}
This proves \eqref{eq:H1bdd}. 
\end{proof}

The main result of this work now follows in a straightforward manner by combining 
the above Lemma \ref{lem:H1bdd} and Corollary \ref{cor:sptinviscid}.

\begin{theorem}\label{thm:kuksupp}
Let $\mu_0$ be a Kuksin measure for the linear inviscid problem \eqref{eq:inviscid}. Then
\begin{itemize}
	\item $\mu_0(L^\infty\cap H^1\cap E)=1$.
	\item $\mu_0=\mathcal{N}(0, Q_0)$, where $Q_0$ is a limit point of $\{Q_\nu\}_{\nu\in(0,1]}$ in the weak operator
	topology.
\end{itemize}
\end{theorem}
The fact that the support of $\mu_0$ is a subset of $L^\infty$ is discussed briefly in Remark \ref{rmk:Linf} below. 
Also, $\mu_0$ is Gaussian since it is the limit of Gaussian measures. However,
in the general case the (subsequential) convergence of the covariance operators $Q_\nu$ can be only
guaranteed in the weak operator
topology.
In order to deduce further properties of $\mu_0$, such as uniqueness or more information on the support, 
one would have to prove better quantitative estimates on $Q_\nu$. As shown in the following Section
\ref{sec:examples}, this will be possible in a few specific cases in which the operator $\uu\cdot \grad-\nu\Delta$
or, equivalently, the evolution semigroup $S_\nu(t)$, is better understood.

\begin{remark} \label{rmk:Linf}
That the support of $\mu_0$ is a subset of $L^\infty$ follows from a variant of \cite{GSV13}*{Theorem 4.2}, in
which an instantaneous parabolic regularization from $L^2$ to $L^\infty$ was shown by means of
a Moser type argument. 
The proof applies to linear advection-diffusion equations with divergence-free velocity field \cite{GSV13}*{Remark 4.4} and fractional dissipation \cite{GSV13}*{Remark 4.5}. 
In \cite{GSV13}, two-dimensionality is used only to avoid the vortex-stretching term that would arise in the three-dimensional Navier-Stokes equations. 
For scalar, linear advection-diffusion equations, the generalization to higher dimensions is similar to the case of fractional dissipation. Both are based on restricting $2^\ast$ in \cite{GSV13}*{Equation (4.18)} to smaller values. For example, in the case of $-\Delta$ dissipation and dimension $d > 2$, one would need to choose $2^\ast \in (2,\frac{2d}{d-2})$. However, the proof of \cite{GSV13}*{Theorem 4.2} works for any fixed $2^\ast > 2$.   
\end{remark} 

\section{Explicit examples}\label{sec:examples}

In this section we discuss some cases where explicit computations are possible. This allows us to determine some more precise information about the Kuksin measures and, in some cases, characterize them explicitly.    

\subsection{Relaxation enhancing flows}
Our first example concerns a class of flows for which the associated Kuksin measure is trivial. 
The concept of relaxation enhancing flow was introduced in \cite{CKRZ08},
although similar issues were investigated in previous works as well \cites{BHN05,K80,K88,K90,K91}. 

\begin{definition}[Relaxation enhancing] \label{def:Relax}
An incompressible velocity 
field $\uu:\TT^d\to \RR^d$ is called relaxation enhancing if for every $\tau>0$ and $\delta>0$, there exists
$\nu_0=\nu_0(\tau,\delta)$ such that for any $\nu<\nu_0$ and any $f_0\in L^2$ we have
\begin{equation}
\| S_\nu(\tau/\nu)f_0\|_{L^2} <\delta \|f_0\|_{L^2}, \label{ineq:Relax} 
\end{equation}
where $S_\nu(t)$ denotes the semigroup associated to \eqref{eq:viscous}.  
\end{definition} 

The main result of \cite{CKRZ08} shows that relaxation enhancing flows can be identified precisely in terms of spectral properties 
of the linear operator $\uu\cdot \grad$.

\begin{theorem}[\cite{CKRZ08}*{Theorem 2.1}]\label{thm:CZRZ}
A Lipschitz continuous incompressible flow $\uu$ is relaxation enhancing 
if and only if the operator $\uu\cdot \grad$
has no eigenfunctions in $H^1$ other than the zero function.
\end{theorem}
In particular, weakly mixing flows \cites{F02,F06,K53,S67} -- flows such that $i \uu \cdot \grad$ has purely continuous spectrum -- are relaxation enhancing. 
Theorem \ref{thm:kuksupp} above shows that if a flow is relaxation enhancing, then there exists a \emph{unique} Kuksin measure and it is simply a single atom of unit mass at zero (in fact, this is true of all invariant measures satisfying $\mu(H^1) = 1$). This is because Theorem \ref{thm:CZRZ} implies that $E=\{0\}$.
However, due to the explicit estimate on $S_\nu(t)$ available from \eqref{ineq:Relax}, for relaxation enhancing flows we write a direct proof of the result by characterizing the covariance of the unique invariant measures $\mu_\nu = \mathcal{N}(0,Q_\nu)$, for $\nu > 0$. 
This proof will also generalize to some further examples. 

\begin{theorem} \label{thm:RelaxKuksin}
Let $\uu$ be a relaxation enhancing flow. Then $\delta_0$, the Dirac mass centered at zero, is the unique Kuksin measure for the linear inviscid evolution $S(t)$.
\end{theorem}
\begin{proof}
As discussed above in Section \ref{sub:Markov}, for every $\nu > 0$ the unique invariant measure for \eqref{eq:SPDEviscous} is a Gaussian $\mathcal{N}(0,Q_\nu)$
with covariance operator $Q_\nu$ given by
\begin{align*}
Q_\nu = \nu\int_0^\infty S_\nu(t)\Psi \Psi^\ast S_\nu(t)^\ast \dd t.
\end{align*}
Note that in view of the structure of $\Psi$ in \eqref{eq:SPDEviscous}, $\Psi \Psi^\ast$ is the operator 
\begin{align*}
\Psi \Psi^\ast \varphi= \sum_{k \in \Naturals} \psi_k^2 \inp{e_k,\varphi} e_k.  
\end{align*}
We proceed to show that $\|Q_\nu \|_{L^2 \rightarrow L^2} \to 0$ as $\nu \to 0$, which immediately yields the desired result.
Since $\uu$ is relaxation enhancing, by Definition \ref{def:Relax}, for all $\delta,\tau>0$, there exists $\nu_0>0$ sufficiently small such that
for all $\nu< \nu_0$
\begin{align}
\norm{S_\nu(\tau/\nu)}_{L^2 \to L^2} < \delta. \label{ineq:SnuRE}
\end{align}
Since $\norm{T}_{L^2 \to L^2} = \norm{T^\ast}_{L^2 \to L^2}$ for all bounded operators  $T:L^2 \to L^2$, we have that \eqref{ineq:SnuRE} holds also for $S_\nu(t)^\ast$. 
We also have the straightforward estimate from the heat equation which holds regardless of the velocity field $\uu$ (as long as it is incompressible), 
\begin{align}
\norm{S_\nu(t)}_{L^2 \to L^2} \leq \e^{-\nu \lambda_1 t}, \qquad \forall t\geq 0,\label{ineq:trivSnu}
\end{align}
where $\lambda_1$ is the first (non-zero) eigenvalue of the Laplacian. In particular, $S_\nu(t)$ is 
a contraction, and the estimate \eqref{ineq:SnuRE} will propagate at later times as well, namely,
for any $\tau$ and $\delta$ there exists a $\nu_0 = \nu_0(\delta, \tau) > 0$ such that
\begin{align}
\norm{S_\nu(t)}_{L^2 \to L^2} < \delta, \qquad \forall t\geq \frac{\tau}{\nu}, \label{ineq:SnuRE2}
\end{align}
for all $\nu < \nu_0$.  
Then, for any $\varphi \in L^2$ with $\|\varphi\|_{L^2} = 1$ we have 
\begin{align*}
\norm{Q_\nu\varphi}_{L^2} & \leq \nu\int_0^\infty \norm{S_\nu(t)\Psi \Psi^\ast S_\nu(t)^\ast \varphi}_{L^2} \dd t \\ 
& \leq \nu \| \Psi\|^2 \int_0^{\tau /\nu}\norm{S_\nu(t)}_{L^2 \to L^2}^2 \dd t 
+\nu \| \Psi\|^2 \int_{\tau /\nu}^\infty\norm{S_\nu(t)}_{L^2 \to L^2}^2\dd t.
\end{align*}
Using \eqref{ineq:trivSnu}-\eqref{ineq:SnuRE2} we then infer that 
\begin{align}
\norm{Q_\nu\varphi}_{L^2} & \leq \nu \| \Psi\|^2  \int_0^{\tau /\nu} \e^{-2\nu \lambda_1 t}  \dd t 
+\nu \| \Psi\|^2 \int_{\tau /\nu}^\infty \e^{-\nu \lambda_1t} \norm{S_\nu(t)}_{L^2 \to L^2}  \dd t \nonumber \\ 
&\leq \nu \| \Psi\|^2 \int_0^{\tau /\nu} \e^{-2\nu \lambda_1 t}  \dd t  +\delta\nu \| \Psi\|^2
 \int_{\tau /\nu}^\infty \e^{-\nu \lambda_1t} \dd t \nonumber \\ 
& \leq \frac{\| \Psi\|^2}{2\lambda_1}\left(1-\e^{-2\lambda_1\tau}\right) + \frac{\delta \| \Psi\|^2}{\lambda_1}\e^{-\lambda_1\tau}. \label{eq:QnuRelaxEst}
\end{align}
Fix $\eps>0$ arbitrary and choose $\tau$ such that $1 - \e^{-2\lambda_1\tau} \leq \eps$ and $\delta < \eps$. 
Then by \eqref{ineq:SnuRE}, there exists an $\nu_0 = \nu_0(\eps)$, 
\begin{align*}
\norm{Q_\nu\varphi}_{L^2}   \leq  \frac{3 \| \Psi\|^2}{2 \lambda_1}\eps, \qquad \forall \nu<\nu_0.
\end{align*}
The norm estimate on $Q_{\nu}$ follows: for all $\eps>0$, there exists a $\nu_0$ such that $\nu < \nu_0$ implies 
\begin{align*}
\norm{Q_\nu}_{L^2 \to L^2} \leq  \frac{3\| \Psi\|^2}{2 \lambda_1}\eps, 
\end{align*}
and hence 
\begin{align*}
\lim_{\nu \to 0} \norm{Q_\nu}_{L^2 \to L^2 } = 0. 
\end{align*}
Since the covariance converges in the operator norm to zero it follows that
$$
\lim_{\nu\to0}\mu_\nu=\delta_0, 
$$
completing the proof. 
\end{proof} 

\begin{remark} 
Notice that the quantitative estimate \eqref{ineq:SnuRE} plays a crucial role in the proof of Theorem \ref{thm:RelaxKuksin} described above, and highlights
the usefulness of having a more quantitative understanding of $S_\nu(t)$ for $\nu > 0$.
\end{remark} 

\subsection{General shear flows}
In this section we discuss the very simple example of shear flows in two dimensions (on $\Torus^{2}$ or a more general torus of arbitrary side length, but let us take the former for simplicity): 
\begin{align} 
\uu(x,y)= 
\begin{pmatrix} u(y) \\ 0\end{pmatrix}.  \label{def:shear}
\end{align}
It will be clear from the proof that analogous results hold also for $d$-dimensional shear flows (for $d\geq 3$) 
with similar proofs. 
To simplify the exposition, we will discuss a relatively nice class of shear flows, rather than concern ourselves with the most general of cases (surely a more general class is possible). 
\begin{definition} \label{def:nondegshear}
We say a shear flow \eqref{def:shear} is \emph{non-degenerate} provided that $u'$ is continuous and that $u'$ vanishes in at most finitely many points. 
\end{definition} 

It will be convenient to write the force in terms of the standard Fourier basis: 
\begin{align*}
\Psi \dd W_t = \sum_{(k,j) \in \Integer_*^2} \psi_{k,j} e_{k,j}\dd W^{k,j}_t, 
\end{align*} 
where $\Integer_*^2=\Integer^2 \setminus \{(0,0)\}$ and 
\begin{align*}
e_{k,j} = \frac{1}{4 \pi^2} \e^{-ikx - ijy} 
\end{align*}
and $\{W^{k,j}_t\}_{(k,j) \in \Integers_*^2}$ are independent Brownian motions. To ensure the force is real-valued, 
we naturally enforce the symmetry conditions
$$
\psi_{k,j} = \overline{\psi_{-k,-j}}, \qquad  W^{k,j}_t  = W^{-k,-j}_t.
$$
Note that despite the apparent coupling, the force can still be written as a sum of independent Brownian motions: 
\begin{align*}
\Psi \dd W_t = \sum_{j> 0}  \left(\psi_{0,j} e_{0,j} +  \psi_{0,-j} e_{0,-j}\right) \dd W^{0,j}_t+ \sum_{(k,j) \in \mathbb{Z}_*^2: k > 0}  \left(\psi_{k,j} e_{k,j} +  \psi_{-k,-j} e_{-k,-j}\right) \dd W^{k,j}_t.
\end{align*} 
We then get the following result.
\begin{theorem} \label{thm:shearchar}
Let $\uu$ be a non-degenerate shear flow in the sense of Definition \ref{def:nondegshear}. Then, for each choice of $\Psi$, the resulting Kuksin measure is given uniquely by a Gaussian $\mathcal{N}(0,Q_0)$ with covariance defined by the following: for any $\varphi \in L^2$,  
\begin{equation}
Q_0 \varphi= \sum_{j \neq 0} \frac{\abs{\psi_{0,j}}^2}{2\abs{j}^2} \l e_{0,j}, \varphi\r e_{0,j}.
\end{equation} 
\end{theorem}  
\begin{remark} 
Notice that even though \eqref{def:shear} is not relaxation enhancing, if $\psi_{0,j} = 0$ for all $j$ then the Kuksin measure is still the Dirac mass $\delta_0$.   
\end{remark}

\begin{proof} 
First we prove that the only $L^2$ eigenfunctions for $L = i\uu \cdot \grad$ with $\uu$ of the form \eqref{def:shear}  are independent of $x$.   
To see this, suppose there existed some $\varphi \in L^2$ such that 
$$
\uu \cdot \grad \varphi = u \partial_x \varphi = i\lambda \varphi,  \qquad\lambda\in \RR
$$
in the sense of distributions but which is not independent of $x$. Taking the Fourier transform with respect to $x$ implies the following almost everywhere in $y$ and $k \neq 0$:  
\begin{align*}
0 = \left(\lambda - k u(y)\right) \hat{\varphi}(k,y). 
\end{align*}
By the hypotheses of non-degeneracy and the mean-value theorem, it follows that $u(y)$ can only take the same value finitely many times, and hence this identity can only be satisfied 
if $\hat{\varphi}(k,y) = 0$ almost everywhere for all $k$ non-zero. 
Consequently, the only possible $H^1$ eigenfunctions are independent of $x$ (almost everywhere) and are all zero eigenfunctions of the operator $u(y) \partial_x$. Therefore, 
$$
E = \set{\varphi \in L^2 : \varphi(x,y) = \varphi(y) \quad \textup{a.e.}}
$$ 
and hence the projection $\Pi_e : L^2 \rightarrow E$ is simply given  by
$$
(\Pi_e \varphi)(x,y) =\int_{\TT} \varphi(x,y) \dd x. 
$$
Moreover, by Theorem \ref{thm:CZRZ}, it follows that if one restricts $\widetilde{L} = L|_{E^\perp}$, with 
$$
E^\perp = \set{\varphi\in L^2 : \int_{\TT} \varphi(x,y) \dd x = 0 \quad \textup{a.e.}},
$$ 
then $\widetilde{L}$ is relaxation enhancing since it has a purely continuous spectrum. 

Next, because $\Pi_e$ and $\Delta$ commute, it follows that
the solution $f^\nu(t)=S_\nu(t)f_0$ of the deterministic viscous problem \eqref{eq:viscous} satisfies 
\begin{align*}
\partial_t \Pi_e f^\nu & = \nu\Delta \Pi_e f^\nu, \qquad \Pi_e f^\nu(0)=\Pi_e f_0
\end{align*}
and so $E$ is an invariant subspace also for $S_\nu(t)$ and not just $S(t)$ -- this is the crucial point of the proof.  
If we denote $\widetilde{S}_\nu(t) = S_{\nu}(t)|_{E^\perp}$, then since $\widetilde{L} = L_{E^\perp}$ is relaxation enhancing,  from Definition \ref{def:Relax} we have that
for every $\tau>0$ and $\delta>0$, there exists
$\nu_0=\nu_0(\tau,\delta)$ such that for any $\nu<\nu_0$ and any $f_0 \in E^\perp$, 
\begin{equation}
\| \widetilde{S}_\nu(\tau/\nu)f_0\|_{L^2} <\delta \|f_0 \|_{L^2}.
\end{equation}
Finally, notice by linearity that
$$
S_{\nu}(t)f_0 = S_{\nu}(t)\Pi_e f_0 + S_{\nu}(t)(I - \Pi_e)f_0,
$$
however, by the above invariants, we also have
 \begin{align*}
\Pi_e S_{\nu}(t)f_0  &= S_{\nu}(t)\Pi_e f_0 = \e^{\nu \partial_{yy}t}\Pi_e f_0 \\ 
(I -\Pi_e)S_{\nu}(t)f_0 & = S_{\nu}(t)(I-\Pi_e)f_0 = \widetilde{S}_{\nu}(t)f_0  
\end{align*}
and the same holds for $S_{\nu}(t)^\ast$. 
Denote for any $\varphi \in L^2$ the operator
\begin{align*}
\Psi \Psi^\ast \varphi = \sum_{(k,j) \in \Integers_*^2} |\psi_{k,j}|^2 \l e_{k,j},\varphi\r e_{k,j}.   
\end{align*}
Therefore, to compute the covariance, for any $\varphi \in L^2$ we have 
\begin{align*}
Q_{\nu} \varphi & = \nu\int_0^\infty S_{\nu}(t) \Psi \Psi^\ast S_{\nu}(t)^\ast \varphi \dd t \\
& = \nu \int_0^\infty \sum_{(k,j) \in \Integer_*^2}|\psi_{k,j}|^2 \l e_{k,j}, S_{\nu}(t)^\ast \varphi\r S_{\nu}(t) e_{k,j} \dd t \\ 
& = \nu \int_0^\infty \sum_{(k,j) \in \Integer_*^2} |\psi_{k,j}|^2 \l S_{\nu}(t) e_{k,j},\varphi\r   S_{\nu}(t) e_{k,j} \dd t \\   
& = \nu\sum_{j\neq 0} \int_0^\infty |\psi_{0,j}|^2 \l S_{\nu}(t) e_{0,j},\varphi\r   S_{\nu}(t) e_{0,j} \dd t   
 + \nu \int_0^\infty \sum_{(k,j) \in \Integer^2:k\neq 0} |\psi_{k,j}|^2 \l S_{\nu}(t) e_{k,j},\varphi\r   S_{\nu}(t) e_{k,j} \dd t \\   
& = \nu\sum_{j \neq 0} |\psi_{0,j}|^2 \int_0^\infty \l\e^{\nu t \partial_{yy}} e_{0,j}, \varphi\r \e^{\nu t \partial_{yy}}  e_{0,j} \dd t
 + \nu\int_0^\infty \sum_{(k,j) \in \Integer^2:k\neq 0} |\psi_{k,j}|^2 \l\widetilde{S}_{\nu}(t)e_{k,j},\varphi\r \widetilde{S}_{\nu}(t) e_{k,j} \dd t \\ 
& := T_1 \varphi + T_2^\nu \varphi.
\end{align*}
The first term, $T_1$, is independent of $\nu$. 
Indeed, since $e_{0,j}$ are eigenfunctions of the heat operator: 
\begin{equation} \label{eq:charT1}
T_1 \varphi  = \nu\sum_{j\neq 0} |\psi_{0,j}|^2 \int_0^\infty \e^{-2 \nu |j|^2 t} \l e_{0,j}, \varphi\r e_{0,j} \dd t 
= \sum_{j\neq 0} \frac{|\psi_{0,j}|^2}{2|j|^2} \l e_{0,j},\varphi\r e_{0,j}. 
\end{equation}
On the other hand, because $\widetilde{S}_\nu(t)$ is relaxation enhancing, the latter term is estimated precisely as in \eqref{eq:QnuRelaxEst} in the proof of Theorem \ref{thm:RelaxKuksin}. Hence, we may deduce as above that 
\begin{align*}
\lim_{\nu \rightarrow 0}\norm{T^\nu_2}_{L^2 \rightarrow L^2} = 0,
\end{align*}
and therefore
\begin{align*}
\lim_{\nu \rightarrow 0}\norm{Q_\nu - T_1}_{L^2 \rightarrow L^2} = 0, 
\end{align*}
completing the proof. 
\end{proof} 

\subsection{Non-degenerate Cellular flows} 
In this section we discuss one last example in which we can get some additional regularity and other kinds of information due to the rigidity of $E$, even if we cannot determine  the invariant measures precisely. 
\begin{definition}
For a smooth streamfunction $\psi$, we say $\uu = \grad^\perp \psi$ is a \emph{non-degenerate cellular flow} if $\Torus^2$ can be tiled with a finite number of open, disjoint, curvilinear polygons $\mathcal{P}_i$ (referred to as \emph{cells}) whose boundaries are smooth except at the vertices, such that the following holds:  
\begin{itemize} 
\item $\displaystyle\Torus^2 = \bigcup_i \overline{\mathcal{P}_i}$; 
\item inside each polygon $\mathcal{P}_i$, there is a unique fixed 
point $\boldsymbol{x}^e_i \in \mathcal{P}_i$ and we assume that the 
other level curves $C_i(z) = \set{\boldsymbol{x} \in \mathcal{P}_i: \psi(\boldsymbol{x}) = z}$ 
are smooth curves which are diffeomorphic to concentric circles away from the edges of $\mathcal{P}_i$, that is, there exists a homeomorphism $\mathcal{M}_i:\mathcal{P}_i \rightarrow \mathbb D$ (where $\mathbb D$ denotes the unit disk) which is a diffeomorphism away from the edges of $\mathcal{P}_i$ such that there is some strictly monotone function $r(z)$ for which we have  $\mathcal{M}_i(C_i(z)) = \set{\boldsymbol{x} \in \mathbb D: \abs{\boldsymbol{x}} = r(z)}$; 
\item the vertices of the polygons, denoted $\boldsymbol{x}^h_j$, are fixed points; 
\item the edges of the polygons are smooth streamlines which form a simply-connected network of heteroclinic connections between the vertices $\boldsymbol{x}^h_j$; 
\item the following non-degeneracy condition holds on each polygon:
\begin{align*}
\partial_z \left(\int_{C_i(z)} \frac{1}{\abs{\grad \psi(\boldsymbol{x})}}\dd\ell \right) = 0
\end{align*}
in at most finitely many points. 
\end{itemize} 
\end{definition}
 
\begin{center}
\begin{figure}[hbpt] 
\includegraphics[scale=0.5]{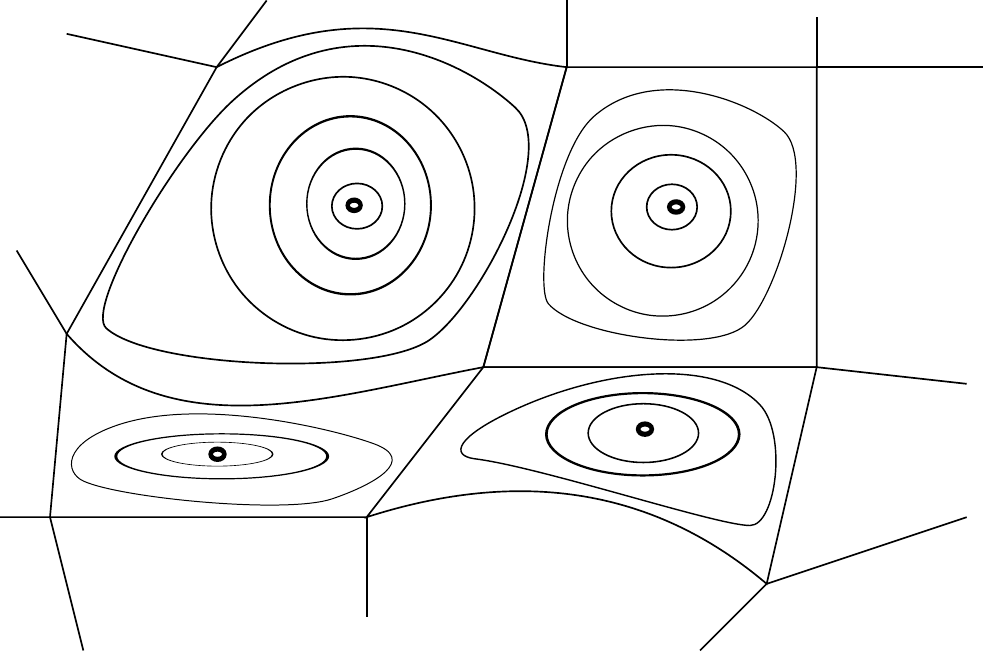}
\caption{A typical cellular flow.} \label{fig:cell}
\end{figure}
\end{center}

In what follows, denote the functions
\begin{align*}
T_{i}(z) = \int_{C_i(z)} \frac{1}{\abs{\grad \psi(\boldsymbol{x})}}\dd\ell, 
\end{align*}
which are the period of the orbit with ``energy'' level $z$.   
Further, denote the set of all fixed points as the disjoint union $F = F_e\cup F_h$, where $F_e = \set{\boldsymbol{x}_1^e,\ldots,\boldsymbol{x}_j^e}$
is the set of fixed points in the interior of the cells and $F_h = \set{\boldsymbol{x}_1^h,\ldots,\boldsymbol{x}_m^h}$ is the set of fixed points comprising the vertices.
We refer to the set $F_h$ together with the heteroclinic connections as the \emph{edge-vertex network}. See Figure \ref{fig:cell} for a schematic of a typical cellular flow. 
As above, denote 
\begin{align*}
E= \overline{\mbox{span}\big\{\varphi\in H^1: i \uu \cdot \grad \varphi= \lambda \varphi, \ \lambda\in \RR \big\}}^{L^2}. 
\end{align*}
Clearly, there are many non-smooth eigenfunctions corresponding to the zero eigenvalue 
(such as functions which are constant over one cell and zero elsewhere), however, the Kuksin 
measures are also supported on $H^1$, which together with the form of the eigenfunctions, 
imposes additional rigidity. In particular we have the following theorem. 
\begin{theorem} \label{thm:cell}
Let $\uu$ be a non-degenerate cellular flow on $\Torus^2$. 
Then
\begin{align*}
E \cap H^1 = \set{\varphi \in H^1: \uu \cdot \grad \varphi = 0 \quad a.e.},   
\end{align*}
and it follows that all $\varphi \in E \cap H^1$ are constant along streamlines, continuous 
on $\Torus^2 \setminus F_e$, $C^{1/2}$ on every compact set $K$ with 
$K \cap F = \emptyset$, and attain the same value everywhere in the edge-vertex network.
Finally, it follows from Theorem \ref{thm:ee} that all Kuksin measures $\mu_0$ associated to the flow $\uu$ satisfy $\mu_0(E \cap H^1) = 1$.
\end{theorem}   
\begin{remark}\label{rmk:ga} 
Theorem \ref{thm:cell} is, heuristically at least, consistent (though much less precise) with the results of Iyer and Novikov in \cite{IN15}, which show that information travels along the edge-vertex network much faster than it travels across streamlines on the interior of the cells and so rapidly homogenizes near the edge-vertex network on time-scales faster than $\nu^{-1}$.   
\end{remark} 
\begin{remark}\label{rmk:fried} 
One can imagine extending Theorem \ref{thm:cell} to wider classes of 2D flows which satisfy suitable non-degeneracy conditions (for example, those studied in \cites{FR02,FW94}). 
\end{remark}
\begin{proof}[Proof of Theorem \ref{thm:cell}]
First note that the streamlines of the flow are precisely the level curves of $\psi$. 
The first step is to prove that all $H^1$ eigenfunctions of $L = i\uu \cdot \grad$ are constant along streamlines almost everywhere 
using a variant of the argument employed in Theorem \ref{thm:shearchar}.  
Suppose $\varphi \in H^1$ is such that 
\begin{align*} 
\grad^\perp \psi \cdot \grad \varphi = i\lambda \varphi,  \qquad\lambda\in \RR
\end{align*} 
in the sense of distributions. 
Consider the cell $\mathcal{P}_j$ and draw a smooth curve $\ell(\tau)$, $\tau\in [0,1]$, 
connecting $\boldsymbol{x}_j^e$ with an edge of the polygon such that $\ell(\tau)$ intersects
each level curve $C_j(z)$ at a single point, which we denote by $\boldsymbol{x}_z$. 
Note this is always possible due to the assumption that the level curves be diffeomorphic to concentric circles: indeed, draw a line from the origin to 
the point $(0,1)$, and then map this line back to $\mathcal{P}_j$ using $\mathcal{M}_j$ and take the resulting curve as $\ell(\tau)$.  
Define
\begin{align*}
\ddt\Phi_t(z)=\grad^\perp \psi (\Phi_t(z)), \qquad \Phi_0(z)=\boldsymbol{x}_z.
\end{align*}
and 
\begin{align*}
h(t,z) = \varphi(\Phi_t(z)), \qquad 
\end{align*} 
defined for $z$ between $\psi(\boldsymbol{x}_j^e)$ and the value of $\psi$ on the edge of the cell; let us denote this range $z\in (z_0,z_1)$.  
Due to the regularity properties of the streamlines, it follows that $\Phi_t(z)$ is a smooth function (for $z$ away from $z_0$ and $z_1$) and hence $h(t,z)$ is $H^1$ away from from $z_0$, $z_1$.  
Hence, for almost every $z \in (z_0,z_1)$, $h(t,z)$ is periodic with period $T_j(z)$ and from the chain rule we have
\begin{align*}
\partial_t h(t,z) = i\lambda h(t,z) 
\end{align*} 
in the sense of distributions. 
For each $z$, it follows that either $\lambda = 0$ (and hence $h(t,z)$ is constant in $t$) or $\lambda$ must be an integer multiple of $T_j(z)$. 
However, by the non-degeneracy hypotheses, $T_j(z)$ can take the same value only finitely many times and hence $h(t,z)$ must be constant in $t$ for almost every $z$. 
It follows then that $\varphi$ is constant along streamlines within the cells.  
Since the edge-vertex network is a measure-zero set, we therefore have that all $H^1$ eigenfunctions are constant along streamlines almost everywhere. 
It further follows that all $H^1$ eigenfunctions have eigenvalue zero. 

Next, let $\varphi \in E \cap H^1$ be arbitrary. 
Since all of the $H^1$ eigenfunctions correspond to the same eigenvalue (zero), it follows that $\varphi$ is itself necessarily an $H^1$ eigenfunction.
We have thus deduced that 
\begin{align*}
E \cap H^1 = \set{\varphi \in H^1: \uu \cdot \grad \varphi = 0 \quad a.e.}.  
\end{align*}
Using diffeomorphisms to locally straighten the streamlines, we see that because $\varphi$ is constant along streamlines, it follows that $\varphi$ must be $C^{1/2}$ away from $F$ (the set of fixed points) by Morrey's theorem 
$H^1_{loc}(\Real) \hookrightarrow C^{1/2}_{loc}(\Real)$.  
That is, $\varphi$ is $C^{1/2}$ in any compact set $K$ such that $K \cap F = \emptyset$.  
Further, by continuity and taking limits along trajectories (along which $\varphi$ is constant) from the interior of the cell, we see that $\varphi$ is constant on the heteroclinic connections between the vertices and takes the same value on any two heteroclinic connections which bound the same cell. 
Therefore, up to a measure zero alteration, we can take $\varphi$ that same value at the vertices and so $\varphi$ is continuous on every compact set which does not intersect $F_e$. 
Finally, by continuity and the connectedness of the edge-vertex network, we see that $\varphi$ must attain the same value everywhere in the entire edge-vertex network. 
\end{proof}

\subsubsection*{Acknowledgments}
The authors would like to thank Yuanyuan Feng, Pierre Germain, Gautam Iyer, Alex Kiselev, Nader Masmoudi, Vlad Vicol, Frederi Viens, and Andrej Zlato\v{s} for helpful discussions. JB was partially supported by NSF grant DMS-1413177. NGH was partially supported by NSF-DMS-1313272. 


\end{document}